\documentclass{amsart}
\usepackage{cases}
\usepackage{}
\usepackage{amsmath}\usepackage{amscd}\usepackage{amssymb}\usepackage{amsfonts}
\newtheorem{theorem}{Theorem}[section]
\newtheorem{lemma}[theorem]{Lemma}

\newtheorem{proposition}[theorem]{Proposition}

\theoremstyle{definition}
\newtheorem{definition}[theorem]{Definition}

\theoremstyle{remark}
\newtheorem{remark}[theorem]{Remark}
\numberwithin{equation}{section}

\begin{document}
\title[Gap theorems for ends of SMMS]
{Gap theorems for ends of smooth metric measure spaces}

\author{Bobo Hua}
\address{Bobo Hua: School of Mathematical Sciences, LMNS, Fudan University, Shanghai 200433, China; Shanghai Center for Mathematical Sciences, Fudan University, Shanghai 200433, China}\email{bobohua@fudan.edu.cn}

\author{Jia-Yong Wu}
\address{Jia-Yong Wu: Department of Mathematics, Shanghai University, Shanghai 200444, China}
\email{wujiayong@shu.edu.cn}

\thanks{}


\dedicatory{}

\date{\today}

\begin{abstract}
In this paper, we establish two gap theorems for ends of smooth metric measure space
$(M^n, g,e^{-f}dv)$ with the Bakry-\'Emery Ricci tensor $\mathrm{Ric}_f\ge-(n-1)$ in a
geodesic ball $B_o(R)$ with radius $R$ and center $o\in M^n$. When $\mathrm{Ric}_f\ge 0$
and $f$ has some degeneration outside $B_o(R)$, we show that there exists an
$\epsilon=\epsilon(n,\sup_{B_o(1)}|f|)$ such that such a space has at most two ends
if $R\le\epsilon$. When $\mathrm{Ric}_f\ge\frac 12$ and $f(x)\le\frac 14d^2(x,B_o(R))+c$
for some constant $c>0$ outside $B_o(R)$, we can also get the same gap conclusion.
\end{abstract}
\maketitle


\section{Introduction and main results}
The Cheeger-Gromoll's splitting theorem \cite{[CG]} states that if a complete
Riemannian manifold $(M^n,g)$ with nonnegative Ricci curvature contains a line,
then $M^n$ is isometric to $N\times\mathbb{R}$ with the product metric, where $N$
is a Riemannian manifold with the Ricci curvature $\mathrm{Ric}(N)\ge0$. As a
consequence, any manifold with nonnegative Ricci curvature has at most two ends.
In \cite{[Cai]}, Cai studied a complete manifold $M^n$ with
$\mathrm{Ric}\ge-(n-1)K$ for some constant $K\geq 0$ in a geodesic ball $B_o(R)$
with radius $R$ and center $o\in M^n$ and $\mathrm{Ric}\ge 0$ outside $B_o(R)$.
He proved that the number of ends of such a manifold is finite and can be
estimated from above explicitly; see also Li and Tam \cite{[LT]} for an
independent proof by a different method. Later, Cai, Colding and Yang \cite{[CCY]}
gave a gap theorem for this class of manifolds, which states that there exists an
$\epsilon(n)$ such that such a manifold has at most two ends if $KR\le \epsilon(n)$.
In this paper we will extend the Cai-Colding-Yang result and get two gap theorems on
smooth metric measure spaces with the Bakry-\'Emery Ricci tensor. Our results
may be useful for understanding the topological information of smooth metric measure
spaces.

Recall that a complete smooth metric measure space (for short, SMMS) is a triple
$(M,g, e^{-f}dv)$, where $(M,g)$ is an $n$-dimensional complete Riemannian manifold,
$dv$ is the volume element of metric $g$, $f$ is a smooth potential function on $M$,
and $e^{-f}dv$ is called the weighted volume element.  On $(M,g,e^{-f}dv_g)$, given
a constant $m>0$, Bakry and \'Emery \cite{[BE]} introduced the $m$-Bakry-\'Emery Ricci tensor
\[
\mathrm{Ric}^m_f:=\mathrm{Ric}+\mathrm{Hess}\,f-\frac{df\otimes df}{m},
\]
where $\mathrm{Ric}$ is the Ricci tensor of $(M,g)$ and $\mathrm{Hess}$ is the
Hessian with respect to the metric $g$. When $m=\infty$, we have the
($\infty$-)Bakry-\'Emery Ricci tensor
\[
\mathrm{Ric}_f=\mathrm{Ric}^{\infty}_f.
\]
If $\mathrm{Ric}_f=\lambda g$ for some $\lambda\in\mathbb{R}$, then
$(M, g, e^{-f}dv)$ is called the gradient Ricci soliton, which is a
generalization of the Einstein manifold.  A Ricci soliton is called
shrinking, steady or expanding, if $\lambda>0$, $\lambda=0$, or
$\lambda<0$, respectively. Gradient Ricci soliton often arises
as a limit of dilations of singularities in the Ricci flow and it plays a
fundamental role in the Ricci flow \cite{[Hami]} and Perelman's resolutions
of the Poincar\'e Conjecture \cite{[P1],[P2],[P3]}.

The Bakry-\'Emery Ricci tensor $\mathrm{Ric}_f$ is linked with the $f$-Laplacian
$\Delta_f:=\Delta-\nabla f\cdot\nabla$ via the generalized Bochner formula
\[
\Delta_f|\nabla u|^2=2|\mathrm{Hess}\,u|^2+2g(\nabla u,\nabla\Delta_f u)
+2\mathrm{Ric}_f(\nabla u, \nabla u)
\]
for $u\in C^\infty(M)$. It plays an important role in the comparison geometry
of SMMSs; see \cite{[WW]}. The Bakry-\'Emery Ricci tensor is also related
to the probability theory and optimal transport. We refer
the reader to \cite{[Lo]}, \cite{[LV]}, \cite{[WW]} for further details.

Lichnerowicz~\cite{[Lic]}, and Wei and Wylie \cite{[WW]} independently extended
the classical Cheeger-Gromoll splitting theorem to a SMMS. It states that if
$(M,g, e^{-f}dv)$ with $\mathrm{Ric}_f\ge 0$ and bounded $f$ contains a line,
then $M=N\times\mathbb{R}$. Fang, Li and Zhang \cite{[FLZ]} showed that the
above splitting result remains true for only an upper bound on $f$. Lim
\cite{[Lim]} observed that the splitting result holds if $\nabla f\to 0$
at infinity. In \cite{[MuW]} Munteanu and Wang proved a splitting result
when $\mathrm{Ric}_f$ has a positive lower bound and $f$ satisfies certain
quadratic growth of distance function. Recently, G. Wu \cite{[Wug]} obtained
splitting results for the gradient Ricci soliton when some integral of the
Ricci curvature along a line is nonnegative. From these results, we see that
the above mentioned manifolds all have at most two ends.

Besides, Wei and Wylie \cite{[WW]} proved that any SMMS with
$\mathrm{Ric}_f>0$ for some bounded $f$ has only one end. The second author
\cite{[Wu]} studied a SMMS with $\mathrm{Ric}_f\ge 0$ outside a compact set
and proved that the number of ends of such a manifold is finite if $f$ has
at most sublinear growth outside the compact set.

Inspired by the gap theorem of manifolds \cite{[CCY]} and the number estimate for
ends of SMMSs \cite{[Wu]}, in this paper we first give a gap theorem for ends of a
smooth metric measure space when $\mathrm{Ric}_f\ge 0$ and $f$ has some degeneration
outside a compact set.
\begin{theorem}\label{Main}
Let $(M,g,e^{-f}dv)$ be an $n$-dimensional complete smooth metric measure
space. Fix a point $o\in M$ and $0<R<1$. Suppose $\mathrm{Ric}_f\ge-(n-1)K$
for some constant $K\ge0$ in the geodesic ball $B_o(R)$; outside $B_o(R)$,
suppose $\mathrm{Ric}_f\ge 0$
and
\begin{equation}\label{fcond}
\underset{r\to\infty}{\overline{\lim}}\frac{1}{r^2}\int^r_0f(\sigma(t))dt\le 0
\end{equation}
on any ray $\sigma$, where $r$ is the distance function starting from $\sigma(0)$.
There exists a constant $\epsilon(n,A)$ depending only on $n$ and $A$, where
$A:=\sup_{x\in B_o(1)}|f(x)|$, such that if
\[
\sqrt{K}R\le\epsilon(n, A),
\]
then $(M,g,e^{-f}dv)$ has at most two ends.
\end{theorem}

\begin{remark}
There exist many examples satisfying Theorem \ref{Main}. Let
$M=\mathbb{S}^{n-1}\times\mathbb{R}$ ($n\geq 3$) with the standard
metric and $f(x,t)=f(t)=1+\int_0^t ds\int_0^s\xi(\tau)d\tau$ for
$(x,t)\in \mathbb{S}^{n-1}\times\mathbb{R}$. Here $\xi(t)$ is a smooth even function
on $\mathbb{R}$ such that $\xi(t)=-(n-1), t\in[0,\tfrac{R}{2}]$,
$\xi(t)\ge-(n-1), t\in(\frac{R}{2},R)$ and $0\le\xi(t)\le |t|^{-\alpha}, t\in[R,\infty)$,
for $R>0$ and $\alpha>1$. One easily checks that $f$ is smooth, even and it
satisfies $f''(t)=\xi(t)$. So $\mathrm{Ric}_f\ge-(n-1)$ and $\mathrm{Ric}_f\ge 0$
outside $B_o(R)$. Since $f(t)\le C|t|^{2-\alpha}+C$ for $|t|\gg 1$, $f(0)=1$ and
$A=\sup_{x\in B_o(1)}|f(x)|\ge 1$, they satisfy the conditions in
Theorem \ref{Main} for $K=1$ and $M$ has two ends. Another example is that:
as in \cite{[CCY]}, by applying the metric surgery techniques to manifold $M=\mathbb{S}^1\times\mathbb{R}\times\mathbb{S}^{n-2}$, $n\ge 4$, one can
get an $n$-dimensional complete manifold $M$ of infinite homotopy
type with exactly two ends and with $\mathrm{Ric}\ge-\delta$ and with
$\mathrm{Ric}\ge 0$ outside a small ball. Let $f(x_1, t,x_2)=-t$ on $M$,
and it satisfies \eqref{fcond}. Then $\mathrm{Ric}_f\ge-\delta$ and
$\mathrm{Ric}_f\ge 0$ outside the small ball.
\end{remark}

\begin{remark}
If $f$ grows sublinearly, then \eqref{fcond} automatically holds. If
$\nabla f\to 0$ at infinity, \eqref{fcond} still holds due to Lim \cite{[Lim]}.
If $f$ is constant, the theorem recovers the Cai-Colding-Yang result
\cite{[CCY]}. Theorem 1.1 is obvious suitable to gradient steady Ricci soliton
because the potential $f$ of steady gradient Ricci solitons is negative
linear outside a compact set. It is an interesting question if one can weaken
the assumption of $f$ such that it is suitable to the quadratic growth of $f$
on gradient shrinking Ricci solitons.
\end{remark}
\begin{remark}
When we say that $E$ is an end of the manifold $M$ we mean
that it is an end with respect to some compact subset of the manifold.
If $R_1\le R_2$, then the number of ends with respect to $B_o(R_1)$
is less than the number of ends with respect to $B_o(R_2)$. So we assume
the radius $R<1$ in the theorem seems to be sensible.
\end{remark}

We can apply a similar argument to get a gap theorem under the conditions of
$\mathrm{Ric}^m_f$ (without any assumption on $f$). It
states that when $\mathrm{Ric}^m_f\ge-(n-1)K$ in $B_o(R)$ and $\mathrm{Ric}^m_f\ge 0$
outside $B_o(R)$, there exists an $\epsilon=\epsilon(n+m)$ depending only on $n+m$
such that $M$ has at most two ends if $\sqrt{K}R\le\epsilon$.

Furthermore, inspired by the Munteanu-Wang splitting theorem \cite{[MuW]}, we may give another
gap theorem for ends when $\mathrm{Ric}_f$ has a positive lower bound and $f$ grows
quadratically outside a compact set. More precisely, after a suitable scaling
of the metric, we may in fact assume $\mathrm{Ric}_f\ge-(n-1)$ in a ball and get that
\begin{theorem}\label{Main2}
Let $(M,g,e^{-f}dv)$ be an $n$-dimensional complete smooth metric measure space.
Fix a point $o\in M$ and $0<R<1$. Suppose $\mathrm{Ric}_f\geq-(n-1)$ in the
geodesic ball $B_o(R)$; outside $B_o(R)$, suppose $\mathrm{Ric}_f\ge\frac 12$
and $f(x)\le\frac 14 d^2(x,B_o(R))+c$ for some constant $c>0$. There exists a
constant $\epsilon(n,A)$ depending only on $n$ and $A$, where
$A:=\sup_{x\in B_o(1)}|f(x)|$, such that if
\[
R\le\epsilon(n, A),
\]
then $(M,g,e^{-f}dv)$ has at most two ends.
\end{theorem}

\begin{remark}
We give an example satisfying Theorem \ref{Main2}.
Let $M=\mathbb{S}^{n-1}\times\mathbb{R}$ ($n\ge 3$) with the standard metric.
For $0<R<1$, let $f(x,t)=f(t)=1+a+bt+\int_0^t ds\int_0^s\eta(\tau)d\tau$ for
$(x,t)\in \mathbb{S}^{n-1}\times\mathbb{R}$, where $a, b$ are chosen satisfying
$f(R)=1, f'(R)=0$. Here $\eta(t)$ is a smooth even function on $\mathbb{R}$
such that $\eta(t)=-(n-1)$, $t\in[0,\tfrac{R}{2}]$, $\eta(t)\ge-(n-1)$,
$t\in(\tfrac{R}{2},R)$ and $\eta(t)=\tfrac 12$, $t\in[R,\infty)$.  One easily
checks that $f$ is smooth, even and $f''(t)=\eta(t)$. So $\mathrm{Ric}_f\ge-(n-1)$ and
$\mathrm{Ric}_f\ge \tfrac 12$ outside $B_o(R)$. Since $f(t)=\frac14(|t|-R)^2+1$ for
$|t|\ge R$, $f(R)=1$ and $A\ge 1$, they satisfy the conditions of Theorem
\ref{Main2} and $M$ has two ends.
\end{remark}

\begin{remark}
If the assumption $R\le\epsilon(n, A)$ in Theorem \ref{Main2} is removed, we can show that
the number of ends for such a space is finite. Moreover, we can provide an
explicit upper bound for the number; see Appendix of the paper.
\end{remark}

We would like to point out that Munteanu and Wang systematically studied the
number of ends on gradient Ricci solitons. In \cite{[MuWa11]} they proved that
any nontrivial steady gradient Ricci soliton has only one end. In \cite{[MuWa12]}
they showed that the expanding gradient Ricci soliton $\mathrm{Ric}_f=-\tfrac 12$ with scalar
curvature $\mathrm{S}\ge-\frac{n-1}{2}$ has at most two ends. They also considered
a similar problem for SMMS with $\mathrm{Ric}_f\ge-\tfrac 12$. In \cite{[MuWa14e]}
they proved that any shrinking K\"ahler gradient Ricci soliton has only one end.
Recently, Munteanu, Schulze and Wang \cite{[MSW]} showed that the number of ends
is finite on shrinking gradient Ricci soliton when the scalar curvature satisfies
certain scalar curvature integral at infinity.

The proof of our theorems uses the argument of Cai-Colding-Yang \cite{[CCY]}
and it relies on a Wei-Wylie's weighted Laplacian comparison \cite{[WW]}
and geometric inequalities for two different ends (see Lemmas \ref{inequa}
and \ref{inequa2}), which are derived by locally analyzing splitting
theorems. We would like to point out that Cai-Colding-Yang's proof depends
on a delicate constructional function $G(r)$, which satisfies certain
Laplacian equation with the Dirichlet boundary condition. In our case,
the function $G(r)$ constructed in Proposition \ref{pro} does not satisfy the
$f$-Laplacian equation, but it is sufficient to deduce our desired results.

The paper is organized as follows. In Section \ref{sec2}, we give some basic
concepts and results on SMMSs. In Section \ref{sec3},
we apply the weighted Laplacian comparison and geometric inequalities for ends
in Section \ref{sec2} to prove our theorems. In Appendix, we give an upper bound
for the number of ends of a class of SMMSs.


\section{Preliminary} \label{sec2}
In this section, we introduce some results about SMMSs,
which will be used in the proof of our results. We first
recall a weighted Laplacian comparison due to Wei and Wylie \cite{[WW]}.
\begin{lemma}\label{comp}
Let $(M,g,e^{-f}dv)$ be an $n$-dimensional complete smooth metric measure
space with a base point $o\in M$. If $\mathrm{Ric}_f\geq-(n-1)K$ for some
constant $K>0$, then
\[
\Delta_f(r)\le(n+4A-1)\sqrt{K}\coth(\sqrt{K}r)
\]
along any minimal geodesic segment $r$ from $o$, where $A=A(o,r)=\sup_{x\in B_o(r)}|f(x)|$.
\end{lemma}

We also have a weighted volume comparison of Wei and Wylie \cite{[WW]}.
The weighted volume is denoted by $V_f(B_x(R)):=\int_{B_x(R)}e^{-f}dv$.
\begin{lemma}\label{vcomp}
Let $(M,g,e^{-f}dv)$ be an $n$-dimensional complete
smooth metric measure space. If $\mathrm{Ric}_f\ge-(n-1)K$ for some
constant $K>0$, then
\begin{equation}\label{volcomp}
\frac{V_f(B_x(r_2))}{V_f(B_x(r_1))}
\leq\frac{\int^{r_2}_0(\sinh^{n-1+4A}\sqrt{K}\,t)dt}{\int^{r_1}_0(\sinh^{n-1+4A}\sqrt{K}\,t)dt}
\end{equation}
for any $x\in M$ and $0<r_1<r_2$, where $A=A(x,r_2)=\sup_{y\in B_x(r_2)}|f(y)|$.
\end{lemma}

Then we recall some definitions of geometric quantities such as line, ray, end and
asymptotic ray on Riemannian manifolds.
\begin{definition}
On a complete Riemannian manifold $(M,g)$, we say that a geodesic
$\gamma:(-\infty,+\infty)\to M$ is called \emph{line} if
\[
d(\gamma(s),\gamma(t))=|s-t|
\]
for all $s$ and $t$. We say that a geodesic $\gamma:[0,+\infty)\to M$ is called
\emph{ray} if
\[
d(\gamma(0),\gamma(t))=t
\]
for all $t>0$. As we all known if $M$ is complete noncompact, it must contain a ray.
\end{definition}

\begin{definition}
On a manifold $M$ with a base point $o\in M$, two rays $\gamma_1$
and $\gamma_2$ starting at $o$ are called cofinal if for any $R>0$ and
$t>R$, $\gamma_1(t)$ and $\gamma_2(t)$ lie in the same component of
$M\setminus B_o(R)$. An equivalent class of cofinal rays is called an
\emph{end} of $M$. In this paper we let $[\gamma]$ be the class
of the ray $\gamma$.
\end{definition}
One readily checks that this definition is independence of the base point $o$
and the complete metric on manifold $M$. Thus, the number of ends is a topological
invariant of $M$.

Next we recall the definition of the Busemann function and its properties on a
complete SMMS $(M,g,e^{-f}dv)$. The \emph{Busemann function}
associated to each ray $\gamma\subset M$ is defined by
\[
b_{\gamma}(x):=\lim_{t\to \infty}(d(x,\gamma(t))-t).
\]
By the triangle inequality, we know that $b_{\gamma}(x)$ is Lipschitz continuous with Lipschitz
constant $1$ and hence it is differential almost everywhere. At the points where $b_\gamma$
is not smooth we interpret the $f$-Laplacian in the following sense of barriers.

\begin{definition}
A lower barrier for a continuous function $h$ at the point $p\in M$ is a
$C^2$ function $h_p$, defined in a neighborhood $U$ of $p$, such that
\[
h_p(p)=h \quad\mathrm{and}\quad h_p(x)\leq h(x), \quad x\in U.
\]
A continuous function $h$ on $M$ satisfies $\Delta_f h\geq a$ at
$p$ in the barrier sense, if for every $\epsilon>0$, there exists a lower barrier function
$h_{p,\epsilon}$ at $p$ such that $\Delta_f h_{p,\epsilon}\geq a-\epsilon$.
A continuous function $h$ satisfies $\Delta_f h\leq a$ in the barrier sense
is similarly defined.
\end{definition}

\begin{definition}
For a fixed point $p\in M$, let $\alpha(t)$ be a minimal geodesic from
$p$ to ray $\gamma(t)$. As $t\to \infty$, $\alpha(t)$ has a convergent
subsequence which converges to a ray at $p$. Such a ray is called an
\emph{asymptotic ray} to $\gamma(t)$ at $p$.
\end{definition}

For a line $\gamma$ in $M$, there exist rays $\gamma^{+}:[0,\infty)\to M$
by $\gamma^{+}(t)=\gamma(t)$ and $\gamma^{-}:[0,\infty)\to M$ by
$\gamma^{-}(t)=\gamma(-t)$. Similar to the above procedure, we can let
$b^{+}_{\gamma}$(or $b^{-}_{\gamma}$, respectively) be the associated Busemann
function of $\gamma^{+}$(or $\gamma^-$, respectively).

\vspace{.1in}

Next we will introduce two geometric inequalities for two different ends
under two types of curvature assumptions, which are important
in the proof of Theorems \ref{Main} and \ref{Main2}. On one hand, recall that
Fang, Li and Zhang \cite{[FLZ]} proved a Cheeger-Gromoll splitting theorem when
$\mathrm{Ric}_f\ge 0$ and $f$ satisfy some degeneration condition. As in \cite{[Wu]},
we can easily apply the Fang-Li-Zhang arguments locally and get that
\begin{lemma}\label{splilem}
Let $N$ be the $\delta$-tubular neighborhood of a line $\gamma$ on
$(M,g,e^{-f}dv)$. Suppose that from every point $p$ in $N$, there
are asymptotic rays to $\gamma^{\pm}$ such that $\mathrm{Ric}_f\ge 0$ and
\eqref{fcond} on both asymptotic
rays. Then through every point in $N$, there exists a line $\alpha$
such that
\[
b_\gamma^+(\alpha^+(t))=t,\quad b_\gamma^-(\alpha^-(t))=t.
\]
\end{lemma}

\begin{proof}[Sketch proof of Lemma \ref{splilem}]
The proof is the same as the argument of Lemma 2.8 in \cite{[Wu]} and we include
it for the completeness. For any point $p\in N$, by \cite{[FLZ]} we firstly prove
that the two asymptotic rays to $\gamma^{\pm}$ are uniquely determined at $p$
and form a line $\gamma_p$. Then we can prove that Busemann functions
$b_\gamma^{\pm}$ at $p$ with $b_\gamma^{+}+b_\gamma^{-}=0$ are smooth $f$-harmonic
functions in the barrier sense. Meanwhile, when $\mathrm{Ric}_f\ge 0$ and
\eqref{fcond} hold on $\gamma_p$, from \cite{[FLZ]} we get that
Busemann functions $b_\gamma^{\pm}$ satisfy $|\nabla b_\gamma^{\pm}|=1$ and
$\mathrm{Hess}\,b_\gamma^{\pm}=0$ on $\gamma_p$. Here the restriction of $b_\gamma^{\pm}$
to $\gamma_p$ is a linear function with derivative 1. So we can reparameterize
$\gamma_p$ and the lemma follows.
\end{proof}

As in the proof of Lemma 3.3 in \cite{[Cai]}, we are able to apply Lemma \ref{splilem} to prove
the following property about ends. We refer the reader to Lemma 3.1 and Proposition
3.2 in \cite{[Wu]} for the detailed discussion.
\begin{lemma}\label{inequa}
Under the same assumptions of Theorem \ref{Main}, $M$ cannot admit a line
$\gamma$ satisfying $d(\gamma(t),B_o(R))\geq |t|+2R$ for all $t$. Moreover,
if $[\gamma_1]$ and $[\gamma_2]$ are two different ends of $M^n$, then for
any $t_1, t_2\ge 0$,
\begin{equation}\label{endineq}
d(\gamma_1(t_1),\gamma_2(t_2))>t_1+t_2-6R.
\end{equation}
\end{lemma}
\begin{remark}
In Proposition 3.2 in \cite{[Wu]}, we only proved \eqref{endineq}
when $t_1,t_2\ge 3R$. Checking the previous proof, we easily
see that \eqref{endineq} in fact holds for all $t_1, t_2\ge 0$.
\end{remark}

On the other hand, recall that Munteanu and Wang \cite{[MuW]} proved another Cheeger-Gromoll
type splitting theorem when $\mathrm{Ric}_f\ge\frac 12$ and $f$ has certain quadratic
growth. As in the preceding argument, we easily get the following result by
analyzing the Munteanu-Wang's proof locally.
\begin{lemma}\label{splilem2}
Let $N$ be the $\delta$-tubular neighborhood of a line $\gamma$ on
$(M,g,e^{-f}dv)$. Suppose that from every point $p$ in $N$, there
are asymptotic rays to $\gamma^{\pm}$ such that $\mathrm{Ric}_f\ge\frac 12$ and
$f(x)\le\frac 14 d^2(x,B_o(1))+c$ for some constant $c>0$ on both asymptotic
rays. Then through every point in $N$, there exists a line $\alpha$
such that
\[
b_\gamma^+(\alpha^+(t))=t,\quad b_\gamma^-(\alpha^-(t))=t.
\]
\end{lemma}

Using Lemma \ref{splilem2}, we can get a geometric inequality about two
different ends along the above similar argument.
\begin{lemma}\label{inequa2}
Under the same assumptions of Theorem \ref{Main2}, $M$ cannot admit a line
$\gamma$ satisfying $d(\gamma(t),B_o(R))\geq |t|+2R$ for all $t$. Moreover,
if $[\gamma_1]$ and $[\gamma_2]$ are two different ends of $M^n$, then for
any $t_1, t_2\ge 0$,
\[
d(\gamma_1(t_1),\gamma_2(t_2))>t_1+t_2-6R.
\]
\end{lemma}


\section{Gap theorems} \label{sec3}
In this section we will prove Theorems \ref{Main} and \ref{Main2}.
We start with an important proposition, which will be used in our proof of theorems.
For the convenient discussion, we assume that $(M,g,e^{-f}dv)$ has
$\mathrm{Ric}_f\ge-(n-1)$ by scaling the metric.
\begin{proposition}\label{pro}
Let $(M,g,e^{-f}dv)$ be an $n$-dimensional complete smooth metric measure
space with a base point $o\in M$. Assume that $\mathrm{Ric}_f\ge-(n-1)$ and
let $A:=\sup_{x\in B_o(1)}|f(x)|$. There exist an $\epsilon=\epsilon(n,A)$
and a $\delta=\delta(n,A)$ such that
\[
u(x)<2-2\delta-12\epsilon
\]
for all $x\in S_o(1-\delta):=\{x\in M|d(x,o)=1-\delta\}$ if $u:M^n\to \mathbb{R}$
is a continuous function satisfying the following properties:
\begin{itemize}
\item[(i)] $u(o)=0$,
\item[(ii)] $u\ge-6\epsilon$,
\item[(iii)] $\sup_{x\neq y}|u(x)-u(y)|/d(x,y)\le 2$,
\item[(iv)] $\Delta_f u\le 2(n+4A-1)$ in the barrier sense.
\end{itemize}
\end{proposition}

\begin{proof}[Proof of Proposition \ref{pro}]
Let $H(r):=2r+G(r)$, where $G(r)$ is defined as
\[
G(r):=2(n+4A-1)\int^1_r\int^1_t\left(\frac{\sinh s}{\sinh t}\right)^{n+4A-1} dsdt
\]
and $A:=\sup_{x\in B_o(1)}|f(x)|$. We remark that in gereral $G(d(x,o))$ does not
satisfy the $f$-Laplacian equation $\Delta_fG=2(n+4A-1)$. But we observe that $G(1)=0$,
\[
G'(r)=-2(n+4A-1)\int^1_r\left(\frac{\sinh s}{\sinh r}\right)^{n+4A-1} ds,
\]
$G'(1)=0$ and $G'(r)\le 0$. So $H(1)=2$ and $H'(r)>0$ when $r\to 1$. Therefore
there exists a real constant $c$ such that $c\in(0,1)$ and $H(c)<2$. Now we choose
$\delta=\delta(n,A)$ and $\epsilon=\epsilon(n,A)$ such that
\begin{equation}\label{cond1}
0<\delta<\tfrac 12\min\left\{2-H(c), 1-c\right\}
\end{equation}
and
\begin{equation}\label{cond2}
0<\epsilon<\tfrac 16\min\left\{G(1-\delta), 2-H(c)-2\delta\right\}.
\end{equation}

Consider function $v(y):=u(y)-G(d(x,y))$ on the annulus $B_x(1)\setminus B_x(c)$.
By the weighted Laplacian comparison (Lemma \ref{comp}) and $G'(r)\le0$, we compute that
\begin{equation*}
\begin{aligned}
\Delta_fG&=G''(r)|\nabla r|^2+G'(r)\Delta_fr\\
&\ge G''(r)+G'(r)\left[(n+4A-1)\coth r\right]\\
&=2(n+4A-1),
\end{aligned}
\end{equation*}
where we used the definition of $G(r)$ in the last equality. This implies that
\[
\Delta_fv\le 0
\]
in the barrier sense by combining the assumption (iv). By the maximum principle,
$v$ achieves its minimum on the boundary of the annulus $B_x(1)\setminus B_x(c)$.
By \eqref{cond1}, we know that $o$ is an interior point of the domain
$B_x(1)\setminus B_x(c)$. By \eqref{cond2} and the assumption (i), we see that
\[
v(o)=u(o)-G(d(o,x))=-G(1-\delta)<-6\epsilon.
\]
Therefore
there exists a point $z$ on the boundary of the annulus such that
\[
v(z)\le v(o)<-6\epsilon.
\]
On the other hand, on the sphere $S_x(1)$, by the assumption (ii), we get that
\[
v=u-G(1)=u\ge-6\epsilon,
\]
where we used $G(1)=0$. This implies that $z\in S_x(c)$. Combining this with
the assumption (iii), the definitions of $H(r)$ and $v$, and \eqref{cond2}, we
finally get
\[
u(x)\le u(z)+2c=v(z)+H(c)<2-2\delta-12\epsilon
\]
and the result follows.
\end{proof}

We now apply Proposition \ref{pro} to prove Theorem \ref{Main}, by
following the argument of Cai-Colding-Yang \cite{[CCY]}.
\begin{proof}[Proof of Theorem \ref{Main}]
When $K=0$, the theorem easily follows by the Fang-Li-Zhang splitting theorem
\cite{[FLZ]}. Now let $(M,g,e^{-f}dv)$ be as Theorem \ref{Main} with $K=1$. Let
$\epsilon=\epsilon(n,A)$ be as Proposition \ref{pro}. We only need to show that
when $R\le\epsilon (n,A)$, $M^n$ has at most two ends. Suppose the conclusion is
not true. That is, there exists three different ends, denoted by $[\gamma_1]$,
$[\gamma_2]$ and $[\gamma_3]$. We consider the function
\[
u(x):=b_{\gamma_1}(x)+b_{\gamma_2}(x).
\]
We claim that $u(x)$ satisfies four conditions of Proposition \ref{pro}.
Indeed, (i) and (iii) are obvious. By \eqref{endineq} and the triangle inequality,
\begin{equation*}
\begin{aligned}
u(x):&=b_{\gamma_1}(x)+b_{\gamma_2}(x)\\
&=\lim_{t\to \infty}(d(x,\gamma_1(t))-t)+\lim_{t\to \infty}(d(x,\gamma_2(t))-t)\\
&\ge \lim_{t\to \infty}\left(d(\gamma_1(t),\gamma_2(t))-2t\right)\\
&\ge-6R\\
&\ge-6\epsilon,
\end{aligned}
\end{equation*}
which implies that $u$ satisfies (ii). Moreover, by Lemma \ref{comp},
\begin{equation*}
\begin{aligned}
\Delta_f\,u(x)&=\Delta_f\,d(x,\gamma_1(\infty))+\Delta_f\,d(x,\gamma_2(\infty))\\
&\le2(n+4A-1)\lim_{r\to \infty}\coth r\\
&= 2(n+4A-1),
\end{aligned}
\end{equation*}
which implies that $u$ satisfies (iv). Therefore, by Proposition \ref{pro}, we
conclude that
\begin{equation}\label{proineq}
u(\gamma_3(1-\delta))<2-2\delta-12\epsilon.
\end{equation}
On the other hand, by Lemma \ref{inequa}, for any $t>0$,
\[
u(\gamma_3(1-\delta)\ge 2t-12R.
\]
In particular, letting $t=1-\delta$, it follows that
\[
u(\gamma_3(1-\delta))\ge2(1-\delta)-12R\ge2-2\delta-12\epsilon.
\]
This contradicts \eqref{proineq} and hence completes the proof.
\end{proof}

Finally we give an explanation how to prove Theorem \ref{Main2}.
\begin{proof}[Sketch proof of Theorem \ref{Main2}]
We can apply Proposition \ref{pro} and Lemma \ref{inequa2} to prove Theorem
\ref{Main2}. In fact the argument in this case is exactly the same as the proof
of Theorem \ref{Main}. Here we omit the details.
\end{proof}

\section{Appendix}\label{sec4}
In this part we will give a number estimate for ends of a class of SMMSs.
The weight $f$ allows to be certain quadratic growth of
distance function, which improves the growth of $f$ in \cite{[Wu]}.
\begin{theorem}\label{Endest}
Let $(M,g,e^{-f}dv)$ be an $n$-dimensional complete smooth metric measure space.
Fix a point $o\in M$ and $R>0$. Suppose $\mathrm{Ric}_f\ge-(n-1)$ in the geodesic
ball $B_o(R)$ and $\mathrm{Ric}_f\ge\frac 12$ outside $B_o(R)$. If
$f(x)\le\frac 14d^2(x,B_o(R))+c$ for some constant $c>0$ on $M$, then
\begin{equation}\label{number}
N_{R}(M)\le\frac{2(n+4A)}{n+4A-1}\cdot\frac{e^{\frac{17}{2}(n+4A-1)R}}{R^{n+4A}}
\end{equation}
where $N_{R}(M)$ is the number of ends of $M$ with respect to
$B_o(R)$, and $A:=A(R)=\sup_{x\in B_o(25/2\,R)}|f(x)|$.
\end{theorem}

The proof follows by Lemmas \ref{inequa2} and \ref{vcomp} by using the arguments
of \cite{[Cai],[Wu]}. We include it for the readers' convenience.
\begin{proof}[Proof of Theorem \ref{Endest}]
For any a point $o\in M$, let $\gamma_1, \gamma_2, ..., \gamma_k$ be $k$ rays with
$k$ different ends starting from the base point $o$. Then we only need to give an
upper bound of the number $k$.

For a fixed $R>0$, consider the sphere $S_o(4R)$ and let $\{p_j\}$ be
a maximal set of points on $S_o(4R)$ such that the balls $B_{p_j}(R/2)$
are disjoint each other. Clearly, the balls $B_{p_j}(R)$ cover $S_o(4R)$.
Since the set $\{\gamma_i(4R), i=1,2...,k\}$ is contained in $S_o(4R)$,
each $\gamma_i(4R)$ is contained in some $B_{p_j}(R)$.
From Lemma \ref{inequa2} with $t=4R$, we know that each ball $B_{p_j}(R)$
contains at most one $\gamma_i(4R)$, and hence the number of balls is not less than
$k$. Therefore, to estimate an upper bound of $k$,
it suffices to bound the number of balls $B_{p_j}(R/2)$.

By the weighted volume comparison \eqref{volcomp}, using a fact that
$B_{p_j}(R/2)\subset B_o(\tfrac{9}{2}R)\subset B_{p_j}(\tfrac{17}{2}R)$,
we have
\[
V_f\left(B_{p_j}(\tfrac{17}{2}R)\right)
\le\frac{\displaystyle{\int^{\tfrac{17}{2}R}_0\left(\sinh^{n+4\widetilde{A}-1}t\right)dt}}{\displaystyle{\int^{R/2}_0
\left(\sinh^{n+4\widetilde{A}-1}t\right)dt}} V_f\left(B_{p_j}(R/2)\right),
\]
where $\widetilde{A}=\sup_{x\in B_{p_j}\left(\tfrac{17}{2}R\right)}|f(x)|$.
Therefore, the number of balls $B_{p_j}(R/2)$ is no more than
\[
\frac{\displaystyle{\int^{\tfrac{17}{2}R}_0\left(\sinh^{n+4A-1}t\right)dt}}{\displaystyle{\int^{R/2}_0
\left(\sinh^{n+4A-1}t\right)dt}},
\]
where $A=\sup_{x\in B_o\left(\tfrac{25}{2}R\right)}|f(x)|$ because $B_{p_j}\left(\tfrac{17}{2}R\right)\subset B_o\left(\tfrac{25}{2}R\right)$.
Notice that
\[
\frac{\displaystyle{\int^{\tfrac{17}{2}R}_0\left(\sinh^{n+4A-1}t\right)dt}}{\displaystyle{\int^{R/2}_0
\left(\sinh^{n+4A-1}t\right)dt}}
\le\frac{2(n+4A)}{n+4A-1}\cdot\frac{e^{\frac{17}{2}(n+4A-1)R}}{R^{n+4A}}
\]
and the upper estimate follows.
\end{proof}

\vspace{.1in}

\textbf{Acknowledgement}.
The authors thank the anonymous referee for valuable comments and useful suggestions to
improve the presentations of this work. B. Hua is supported by NSFC (No.11831004).

\bibliographystyle{amsplain}

\end{document}